\documentclass[11pt]{amsart}
\usepackage{amssymb, latexsym, amsmath, amsfonts, tikz}

\newtheorem{thm}{Theorem}[section]

\newtheorem{lem}[thm]{Lemma}
\newtheorem{prop}[thm]{Proposition}
\theoremstyle{definition}

\theoremstyle{remark}

\numberwithin{equation}{section}
\theoremstyle{remark}

\newcommand{\mbb}{\mathbb}
\newcommand{\ra}{\rightarrow}

\newcommand{\pa}{\partial}

\newcommand{\sm}{\setminus}

\newcommand{\no}{\noindent}

\newcommand{\cal}{\mathcal}

\newcommand{\abs}[1]{\left\vert#1\right\vert}

\begin{document}
\title{A comparison of two biholomorphic invariants}
\author{Prachi Mahajan, Kaushal Verma}

\address{PM: Department of Mathematics,
Indian Institute of Technology Bombay, Powai,
Mumbai 400076, India}
\email{prachi@math.iitb.ac.in}

\address{KV: Department of Mathematics, Indian Institute of Science, Bangalore 560 012, India}
\email{kverma@iisc.ac.in}

\begin{abstract}
The Fridman invariant, which is a biholomorphic invariant on Kobayashi hyperbolic manifolds, can be seen as the dual of the much studied squeezing function. We compare this pair of invariants by showing that they are both equally capable of determining the boundary geometry of a bounded domain if their boundary behaviour is apriori known. 
\end{abstract}

\maketitle

\section{Introduction}

\no Recall that the squeezing function associated to a bounded domain is a measure of the largest euclidean ball 
contained in all possible holomorphic embeddings of the given domain into the unit ball in $\mbb C^n$. More precisely, for
a bounded domain $D \subset \mbb C^n$ and $p \in D$, let $\cal F$ be the family of all injective holomorphic maps $ f $ 
from $D$ to the unit ball $\mbb B^n \subset \mbb C^n$ that map $p$ to the origin. Let $S_D(p, f)$ be the supremum of those $r > 0$ for
which the image $f(D)$ contains $B^n(0, r)$, the euclidean ball of radius $r$ around the origin in $ \mathbb{C}^n $. The squeezing function $s_D : D \mapsto (0, 1]$
is defined as 
\[
s_D(p) = \sup \{S_D(p, f) : f \in \cal F \}.
\]
That this is a biholomorphic invariant follows from its definition and when $D = \mbb B^n$, it can be checked that $s_D \equiv 1$. Various aspects of $s_D$ have been studied of late but among those that are directly relevant to this note are  its boundary behaviour on some classes of domains (for example \cite{FW}, \cite{KZ}, \cite{N} and \cite{NA}) and conversely, its efficacy in determining some geometric properties of the boundary of the domain if its boundary behaviour is a priori known -- for example, \cite{SK} and \cite{Z2}.

\medskip

It is interesting to note that another biholomorphic invariant, that is dual to the squeezing function in much the same way as the Carath\'{e}odory and Kobayashi metrics are, was defined by Fridman in \cite{Fr1}, \cite{Fr2}. Let us recall its construction: for $X$ a Kobayashi hyperbolic complex manifold of dimension $n$, let $B_X(p, r)$ be the Kobayashi ball around $p \in X$ of radius $r > 0$. Let $\cal R$ be the set of all $r > 0$ such that there is an injective holomorphic map $f : \mbb B^n \ra X$ with $B_X(p, r) \subset f(\mbb B^n)$. Note that 
$\cal R$ is non-empty. Indeed, the hyperbolicity of $X$ implies that the intrinsic topology on it is equivalent to that induced by the Kobayashi metric. Hence, for small $r >0$, the ball $B_X(p, r)$ is contained in a coordinate chart and this shows that there is an injective holomorphic map from the ball into $X$ whose image contains $B_X(p, r)$. The Fridman invariant is
\[
h_X(p) = \inf_{r \in \cal R} \frac{1}{r}
\]
which is a non-negative real-valued function on $X$. This is a biholomorphic invariant since the Kobayashi metric is itself preserved 
by such maps and among other things, Fridman showed that (i) $h_X$ is continuous and that (ii) if $h_X(p) = 0$ for some $p \in X$, then
$X$ is biholomorphic to $\mbb B^n$ in \cite{Fr1}. Other aspects of this invariant such as its boundary behaviour were studied in \cite{MV}.

\medskip

The purpose of this note is (i) to show that much like the squeezing function,  the Fridman invariant can also determine 
the nature of the boundary of a given domain if its boundary behaviour is a priori known and (ii) to localize and provide a
different proof of some of the results in \cite{N} and \cite{SK}. While both (i) and (ii) use the methods of scaling, they 
rely on an observation made in \cite{MV} namely, the convergence of the integrated Kobayashi distance on each scaled domain to that in the limiting domain.
More specifically, refer to Lemma 5.2 and 5.7 of \cite{MV}.

\begin{thm}
Let $D \subset \mbb C^n$ be a bounded convex domain with $C^{\infty}$-smooth boundary. Then $\partial D$ is strongly pseudoconvex if $h_D(z) \ra 0$ as $z \ra \pa D$.
\end{thm}

\noindent The corresponding statement for the squeezing function $s_D(z)$ is already known -- see \cite{Z2} for instance. Note that the boundary $\pa D$ can apriori be of infinite type near $p^0$ in this theorem.

\medskip

A similar result holds for $h$-extendible boundary points.
Before proceeding further, recall that $ p^0 \in \partial D $ is said to be 
an $h$-extendible boundary point if $ \partial D $ is smooth pseudoconvex finite type near $ p^0 $ and the Catlin and the D'Angelo multitypes at $ p^0 $ coincide. The class of $h$-extendible points includes smooth pseudoconvex finite type boundary points in $ \mathbb{C}^2 $, convex finite type points in $ \mathbb{C}^n $ and pseudoconvex finite type boundary points in $ \mathbb{C}^n $ with Levi-form having at most one degenerate eigenvalue. 

\begin{thm}
Let $D \subset \mbb C^n$ be a bounded domain with $p^0 \in \partial D$. Assume that $\pa D$ is $C^{\infty}$-smooth and $h$-extendible near $p^0$. Then $\partial D$ is strongly pseudoconvex near $p^0$ if either $h_D(z) \ra 0$ or $s_D(z) \ra 1$ as $z \ra p^0$.
\end{thm}

\medskip

\noindent Here, it turns out that  $\partial D$ is strongly pseudoconvex near $p^0$ if 
either $h_D(p^j) \ra 0$ or $s_D(p^j) \ra 1$ for a sequence $ p^j $ in $ D $ converging to $ p^0 $ only 
along the inner normal to $ \partial D $ at $ p^0 $. This will be evident from the proof of this theorem.

\medskip

\no For a domain $ D $ in $ \mathbb{C}^n $, $ F_D $ denotes its Kobayashi-Royden infinitesimal metric and
$ d_D $ its integrated Kobayashi distance.

\medskip

\noindent Before proving these theorems, we begin with:

\section{Two Examples}

\begin{lem}
The Fridman invariant for the unit polydisc $ \Delta^n $ is given by
\[
h_{\Delta^n}(z) = 2 \left(\log \left( \frac{\sqrt{n} + 1}{\sqrt{n} - 1}    \right)  \right)^{-1}
\]
for every $z \in \Delta^n$.
\end{lem}

\begin{proof} Since $ \Delta^n $ is homogeneous and the function $ h_{\Delta^n} \left( \cdot \right) $ is a biholomorphic invariant, 
it is enough to compute the explicit formula for $ h_{\Delta^n } $ at the origin. If $ f: \mathbb{B}^n \rightarrow \Delta^n $ is a holomorphic 
imedding such that $ f(0) = 0 $ and
\[
 \Delta^n \left( 0, \frac{e^{2r} -1}{e^{2r} + 1}\right) = B_{\Delta^n} (0,r) \subset f(\mathbb{B}^n),
\]
then it follows from \cite{Alexander} that
\[
 \frac{e^{2r} -1}{e^{2r} + 1} \leq \frac{1}{\sqrt{n}},
\]
or equivalently that,
\[
 \frac{1}{r} \geq 2 \left(  \log \left( \frac{ \sqrt{n} + 1}{ \sqrt{n} -1} \right)\right)^{-1},
\]
which implies that 
\begin{equation} \label{k14}
h_{\Delta^n} \left(0 \right) \geq 2 \left(  \log \left( \frac{ \sqrt{n} + 1}{ \sqrt{n} -1} \right)\right)^{-1}.
\end{equation}
On the other hand, consider $ \psi_0 = i \circ \psi $, where $ \psi $ is an automorphism of $ \mathbb{B}^n $ that preserves the origin and
$ i : \mathbb{B}^n \rightarrow \Delta^n $ is the inclusion map. Then $ \psi_0$ is an imbedding of $ \mathbb{B}^n $ into $ \Delta^n $ satisfying $ \psi_0(0) = 0 $ and
\[
 \Delta^n \left( 0, \frac{1}{\sqrt{n}} \right) =  B_{\Delta^n} \left(0, \frac{1}{2} \log \left( \frac{ \sqrt{n} + 1}{ \sqrt{n} -1}  \right) \right) 
 \subset \psi_0(\mathbb{B}^n),
\]
and hence
\begin{equation} \label{k15}
h_{\Delta^n} \left(0 \right) \leq 2 \left(  \log \left( \frac{ \sqrt{n} + 1}{ \sqrt{n} -1} \right)\right)^{-1}.
\end{equation}
Combining the inequalities (\ref{k14}) and (\ref{k15}) yields the desired expression for $ h_{\Delta^n} $. 

\end{proof}

\begin{lem}
Let $\{p^j\}$ be a sequence in the punctured disc $\Delta \sm \{0\}$ that converges to the origin. Then
\begin{equation} \label{E9}
L_j \le h_{\Delta \sm \{0\}}(p^j) \le U_j
\end{equation}
where
\[
L_j^{-1} = \log \left( 2 \left( -\frac{\pi}{\log \vert p^j \vert} \right)^2 + 1 + \frac{2 \pi}{\log \vert p^j \vert} \sqrt{\left( -\frac{\pi}{\log \vert p^j \vert} \right)^2 + 1} \right) 
\]
and
\[
U_j^{-1} = \log \left(  \left(-\frac{\pi}{\log \vert p^j \vert} \right) + \sqrt{\left(-\frac{\pi}{\log \vert p^j \vert}\right)^2 + 1 } \right).
\]
\end{lem}

\begin{proof} Since $ h_{\Delta \sm \{0\}} (\cdot) $ is a biholomorphic invariant, after composing with an appropriate automorphism of $ \Delta \setminus \{0\} $,
we may assume that each 
$ p^j $ lies in the open interval $ (0,1) $. Consider 
the slit disc $ \Delta \setminus (-1, 0] $, which is a simply connected domain. Choose a conformal map $ f^j $ from the unit disc $ \Delta $  
onto the slit domain $ \Delta \setminus (-1,0] $ such that $ f^j (0) = p^j $. Then 
\begin{equation*}
 B_{ \Delta \setminus \{0\} } \left( p^j, r (p^j) \right) \subset \Delta \setminus (-1,0], 
\end{equation*}
where
\begin{equation} \label{E8}
 r(p^j) = \log \left( - \frac{\pi}{\log p^j} + \sqrt{ \left( - \frac{\pi}{\log p^j}\right)^2 + 1} \right).
\end{equation}
To establish this claim, it suffices to show that
\begin{equation} \label{E7}
d_{\Delta \setminus \{0\} }  \left( p^j, (-1,0) \right) : = \inf_{q \in (-1,0)} d_{\Delta \setminus \{0\} }  \left( p^j, q \right) = r(p^j).  
\end{equation}
To verify this, first recall that the upper half-plane $ \mathbb{H} $ is the universal covering space of the punctured disc $ \Delta \setminus \{0\} $, 
the projection being given by the map
\[
 \mathbb{H} \ni z \mapsto \exp( \iota z ) \in \Delta \setminus \{0\}.  
\]
Hence, for each $ q $ in $ (-1,0) $, 
\begin{equation*}
d_{\Delta \setminus \{0\} }  \left( p^j, q \right) = \inf_{\tilde{q}} d_{\mathbb{H}}  \left( - \iota \log p^j, \tilde{q} \right),
\end{equation*}
where the infimum is taken over all preimages $ \tilde{q} $ of $ q $ under the covering map. Furthermore, the preimage of the interval $ (-1,0) $ is the 
vertical line $ \Re z = \pi $ which is a geodesic in $ \mathbb{H} $. It follows that
\begin{equation*}
d_{\Delta \setminus \{0\} }  \left( p^j, (-1,0) \right)  =  d_{\mathbb{H}} \left( - \iota \log p^j, \{ z \in \mathbb{H} : \Re z = \pi \} \right). 
\end{equation*}
To calculate the right-hand side, observe that there is a unique geodesic 
(namely, the half-circle centred at $ \pi $ and radius $ | - \iota \log p^j - \pi | $) through $ - \iota \log p^j $ 
and orthogonal to the line $ \{ \Re z = \pi \} $. Moreover, $ d_{\mathbb{H}} \left( - \iota \log p^j, \{ z: \Re z = \pi \} \right) $ is the distance from 
$ - \iota \log p^j $ to $  \{ \Re z = \pi \} $ measured along this half-circle. I.e., 
\begin{equation*}
d_{\mathbb{H}} \left( - \iota \log p^j, \{ z \in \mathbb{H} : \Re z = \pi \} \right) = d_{\mathbb{H}} \left( - \iota \log p^j, \pi + \iota | \iota \log p^j + \pi | \right). 
\end{equation*}
The Kobayashi distance between two points $ z , w $ on the upper half-plane is given by 
\begin{equation} \label{E14}
 d_{\mathbb{H}} (z,w) = \log \left( \frac{ |z - \bar{w}| + |z-w| }{|z- \bar{w}| - |z-w| }\right).
\end{equation}
Using the above formulation of the Kobayashi distance on $ \mathbb{H} $, it can be seen that
\begin{equation*}
 d_{\mathbb{H}} \left( - \iota \log p^j, \pi + \iota | \iota \log p^j + \pi | \right) = \log \left( - \frac{\pi}{\log p^j} + \sqrt{ \left( - \frac{\pi}{\log p^j}\right)^2 + 1} \right),
\end{equation*}
and consequently that,
\begin{equation*}
 d_{\Delta \setminus \{0\} } \left(p^j, (-1,0) \right) = \log \left(- \frac{\pi}{\log p^j} + \sqrt{ \left( - \frac{\pi}{\log p^j}\right)^2 + 1} \right), 
\end{equation*}
thereby verifying the equation (\ref{E7}). To summarize, there is a biholomorphic imbedding $ f^j : \Delta \rightarrow \Delta \setminus \{0\} $ with $ f^j(0) = p^j $ and
\begin{equation*}
 B_{ \Delta \setminus \{0\} } \left( p^j, r (p^j) \right) \subset f^j(\Delta) = \Delta \setminus (-1,0], 
\end{equation*}
where $ r(p^j) $ is as defined by equation (\ref{E8}). It follows that 
\[
 h_{\Delta \setminus \{0\} } (p^j) \leq 1/r(p^j),
\]
which gives the upper estimate (\ref{E9}). 

\medskip

\noindent For the lower estimate, the following observations will be needed. Firstly, the punctured disc is complete hyperbolic and hence taut. Moreover, for each $j $, 
$ h_{\Delta \setminus \{0\} } (p^j) > 0 $, and hence there exist a biholomorphic imbedding $ f^j : \Delta \rightarrow \Delta \setminus \{0\} $ with $ f^j(0) = p^j $ 
and
\begin{equation} \label{E10}
  B_{ \Delta \setminus \{0\} } \left( p^j, \frac{1}{h_{\Delta \setminus \{0\}  } (p^j) } \right) \subset f^j(\Delta) \subset \Delta \setminus \{0\}. 
\end{equation}
Secondly, consider the circle centred at the origin and radius $ p^j$,
\[
 C^j = \{ w \in \mathbb{C} : |w| = p^j\},
\]
 and compute 
\begin{equation*}
 \sup_{ q \in C^j} d_{\Delta \setminus \{0\} } \left(p^j, q \right).
\end{equation*}
It turns out that 
\begin{equation} \label{E12}
s(p^j) := \sup_{ q \in C^j} d_{\Delta \setminus \{0\} } \left(p^j, q \right) =  
\log \left( 2 \left( - \frac{\pi}{\log p^j } \right)^2 + 1+ \frac{2 \pi}{\log p^j } \sqrt{ \left( - \frac{\pi}{\log p^j}\right)^2 + 1} \right)  
\end{equation}
Grant this for now. It follows that the circle $ C^j $ is contained in the closure of the Kobayashi ball $ B_{ \Delta \setminus \{0\} } \left( p^j, s(p^j) \right) $.  
This forces that
\begin{equation} \label{E13}
 \frac{1}{h_{\Delta \setminus \{0\}  } (p^j)} \leq s(p^j). 
\end{equation}
Indeed, assume on the contrary that the above inequality does not hold, i.e., there is an $ \epsilon_0 > 0 $ such that
\begin{equation*}
 s(p^j) + \epsilon_0 < \frac{1}{h_{\Delta \setminus \{0\}  } (p^j) }.
\end{equation*}
Then it is immediate that 
\begin{equation} \label{E11}
 C^j \subset  B_{ \Delta \setminus \{0\} } \left( p^j, s(p^j) + \epsilon_0 \right) \subset 
 B_{ \Delta \setminus \{0\} } \left( p^j, \frac{1}{h_{\Delta \setminus \{0\}  } (p^j) } \right).
\end{equation}
Combining (\ref{E10}) and (\ref{E11}) gives
\begin{equation*}
 C^j \subset f^j(\Delta) \subset \Delta \setminus \{0\}.
\end{equation*}
But $ f^j(\Delta) $ is a simply connected sub-domain of the punctured disc and hence it cannot contain any circle centered at the origin. Hence 
we arrive at a contradiction, thereby proving the inequality (\ref{E13}).

\medskip

\noindent The final step is to establish equation (\ref{E12}). It follows from the definition that 
\begin{equation*}
 s(p^j) := \sup_{ q \in C^j} d_{\Delta \setminus \{0\} } \left(p^j, q \right) = 
 \sup_{ q \in C^j} \inf_{\tilde{q}} d_{\mathbb{H} } \left(- \iota \log p^j, \tilde{q} \right), 
\end{equation*}
where the infimum is taken over all preimages $ \tilde{q} $ of $ q $ under the covering projection $ z \mapsto \exp (\iota z) $. Write $ q = p^j \exp ( \iota \theta) $
for $ \theta \in [0, 2\pi) $, so that the right hand side above equals
\begin{equation*}
 \sup_{\theta \in [0, 2\pi)} \inf_{ k \in \mathbb{Z}} d_{\mathbb{H}} \left( - \iota \log p^j, 2 \pi k + \theta - \iota \log p^j \right)
\end{equation*}
A direct computation using the explicit expression (\ref{E14}) for $ d_{\mathbb{H}} (\cdot, \cdot) $ shows that
\begin{equation*}
\inf_{ k \in \mathbb{Z}} d_{\mathbb{H}} \left( - \iota \log p^j, 2 \pi k + \theta - \iota \log p^j \right) = 
\log \left( \frac{ \theta^2 + 2 (\log p^j)^2 + \theta \sqrt{\theta^2 + 4  (\log p^j)^2} }{2 (\log p^j)^2}\right), 
\end{equation*}
so that 
\begin{alignat*}{3}
  \sup_{\theta \in [0, 2\pi)} \inf_{ k \in \mathbb{Z}} d_{\mathbb{H}} \left( - \iota \log p^j, 2 \pi k + \theta - \iota \log p^j \right) 
  = & \log \left( \frac{ 4 \pi^2 + 2 (\log p^j)^2 + 2 \pi \sqrt{4 \pi^2 + 4  (\log p^j)^2} }{2 (\log p^j)^2}\right) \\
  = & \log \left( 2 \left( - \frac{\pi}{\log p^j } \right)^2 + 1+ \frac{2 \pi}{\log p^j } \sqrt{ \left( - \frac{\pi}{\log p^j}\right)^2 + 1} \right),  
 \end{alignat*}
thereby verifying (\ref{E12}). 

\end{proof}

\no Note that both $L_j, U_j \ra +\infty$ as $p^j \ra 0$ which is expected. Thus $h_{\Delta \sm \{0\}}$ blows up near the origin. On the other hand, that 
$h_{\Delta \sm \{0\}} (p^j) \ra 0$ if $\vert p^j \vert \ra 1$ can be seen from the following:

\begin{lem}
Let $D \subset \mbb C$ be a bounded domain with $p^0 \in \pa D$. Assume that $\pa D$ is $C^2$-smooth near $p^0$. Then $h_D(z) \ra 0$ as $z \ra p^0$.
\end{lem}

\begin{proof} Let $ \rho $ be a $ C^2$-smooth local defining function for $ \partial D $ near $ p^0 $ and $ \{ p^j \} $ be a sequence of points in $ D $ 
converging to $ p^0 $. Consider the dilations
\[
 T^j(z) = \frac{z-p^j}{- \rho(p^j)}
\]
and note that the scaled domains $ D^j = T^j (D)$ are given by
\[
 \{ z \in \mathbb{C}: - 1 + 2 \Re \left( \partial \rho(p^j) z \right) - \psi(p^j) O(1) < 0 \}.
\]
near $ T^j(p^0) $. It follows that the sequence of domains $ D^j $ converge in the Hausdorff sense to the half-plane
\[
 D_{\infty} = \{ z \in \mathbb{C}: 2 \Re \left( \partial \rho(p^0) z \right) - 1 < 0 \}.
\]
Note that $D$ supports a local holomorphic peak function at $p_0$ since the boundary $\pa D$ is $C^2$-smooth near it and hence the proof of Theorem 1.1 of \cite{MV} can be adapted to show that 
\[
 h_{D} (p^j) \rightarrow  h_{D_{\infty}} (0).
\]
But $ h_{D_{\infty}} (\cdot) \equiv 0 $ as $ D_{\infty} $ is biholomorphically equivalent to $ \mathbb{B}^n $. Hence  $  h_{D} (p^j) \rightarrow 0 $ as 
$ j \rightarrow \infty $.

\end{proof}

\section{Proof of Theorem 1.1}
\noindent Let $p^0 \in \pa D$. We will study the behaviour of $h_D(z)$ as $z \ra p^0$. The proof of Theorem 1.1 divides into two parts:



\begin{enumerate}
  \item [(i)] $ \partial D $ is of finite type near $ p^0 $, or
  \item [(ii)] $ \partial D $ is of infinite type near $ p^0 $.
\end{enumerate}
 
\noindent It turns out that $ p^0 \in \partial D $ cannot be of infinite type, thereby, ruling out case(ii). 
\medskip 


\noindent {\it Case (i):} let $ p^j $ be a sequence of points in $ D $ converging to $ p^0 $ along the inner normal to $ \partial D $ at $ p^0 $. Since 
$ \lim_{j \rightarrow \infty} h_{D} (p^j) = 0 $ by assumption, there exists a sequence of positive real numbers $ R_j \rightarrow \infty $ and
a sequence of biholomorphic imbeddings $ F^j: \mathbb{B}^n \rightarrow D $ satisfying $ F^j (0)= p^j $ and $ B_D(p^j, R_j) \subset F^j( \mathbb{B}^n) $.

\medskip

\noindent Before going further, let us briefly recall the scaling technique from \cite{Mcneal-1994}. Here and in the sequel, we
write $ z = (z_1, z_2, \ldots, z_{n-1}, z_n ) = ('z,z_n) \in \mathbb{C}^n $ for brevity. By \cite{Yu-1992} there exists a local coordinate
system $ \Phi $ in a neighbourhood of $ p^0 $ such that $ \Phi(p^0) = ('0, 0) $, $ 
\Phi (p^j) = ('0, - \| p^0 - p^j \|) $ for each $ j$ and the domain $ \Phi(D) $ near origin can be written as
\[
\{ ('z,z_n) \in \mathbb{C}^n : 2 \Re z_n + P_0('z) + R(z) < 0 \},
\]
where $P_0$ is a nondegenerate weighted homogeneous polynomial of
degree $1$ with respect to the weights $ \mathcal{M}(\partial D,
0) $, the multitype of $ \partial D $ near the origin, and $ R $ denotes terms of degree at least two.  Define a 
dilation of coordinates by
\[
T^j( z_1, z_2, \ldots, z_{n-1}, z_n ) = \left( \delta_j^{-1} z_1, \delta_j^{-1} z_2, 
\ldots, \delta_j^{-1} z_{n-1}, \delta_j^{-1} z_n \right),
\]
where $ \delta_j = \|p^0 - p^j \| $ for each $j$.
Note that $ T^j \big( ('0, - \delta_j) \big) = ('0, -1) $ for all $j$ and the scaled domains $ D^j = T^j \circ \Phi(D) $ 
converge in the Hausdorff sense to
\[
D_{ \infty} = \big\{ z \in \mathbb C^n : 2 \Re z_n  + P_{0}('z) < 0 \big\}.
\]
Furthermore, it follows from Theorem 1.1 of  \cite{Mcneal-1992} that $ D_{\infty} $ is complete hyperbolic and hence $ D_{\infty} $ is taut.

\medskip 
\noindent Consider the dilated maps
\[
\psi^j := T^j \circ \phi \circ F^j : \mathbb{B}^n \rightarrow D^j.
\]
Note that $ T^j \circ \phi \circ F^j \big( ('0,0) \big) = ('0,-1)$ for each $ j $. Since the domains $ D^j $ are contained in the intersections 
of certain half spaces (see \cite{Gaussier-1997} for details), it follows that the sequence $ \{ T^j \circ \phi \circ F^j  \} $ admits
a subsequence, that will still be denoted by the same indices, that
converges uniformly on compact sets of $ \mathbb{B}^n $ to a
holomorphic mapping $ \psi : \mathbb{B}^n \rightarrow D_{\infty} $.

\medskip

\noindent Then $ \psi $ is a biholomorphism from $ \mathbb{B}^n $ onto $ D_{\infty} $. To establish this, first note that for each $ \epsilon > 0 $,
\begin{equation}
B_{D_{\infty}} \left( ('0,-1), R - \epsilon \right) \subset B_{D^j} \left( ('0,-1), R \right)
\label{k9}
\end{equation}
for all $ R > 0 $ and all $j$ large and this will follow from
\[
\displaystyle\limsup_{j \rightarrow \infty} d_{D^j} \left( ('0,-1), \cdot \right)
\leq d_{D_{\infty}} \left( ('0,-1), \cdot \right).
\]
To verify the above inequality, fix $ q \in D_{\infty}$ and let $ \gamma : [0,1]
\rightarrow D_{\infty} $ be a piecewise $C^1$-smooth path in $
D_{\infty} $ such that $ \gamma(0) = ('0,-1), \gamma(1) = q$ and
\[
\int_0^1 { F_{D_{\infty}} \big( \gamma(t), \dot{\gamma}(t) \big)
dt} \leq d_{D_{\infty}} \left( ('0,-1), q \right) + \epsilon/2.
\]
Since the trace of $ \gamma $ is relatively compact in $
D_{\infty}$, it follows that the trace
of $\gamma $ is contained uniformly relatively compactly in $ D^j
$ for all large $j$. It follows from Lemma 6.2 of \cite{MV} that
\[
\int_0^1 {F_{D^j} \big( \gamma(t), \dot{\gamma}(t) \big) dt}
\leq \int_0^1 { F_{D_{\infty}} \big( \gamma(t), \dot{\gamma}(t)
\big) dt} + \epsilon/2 \leq d_{D_{\infty}} \left( ('0,-1), q \right) + \epsilon.
\]
Consequently,
\[
d_{D^j} \left( ('0,-1), q \right) \leq \int_0^1 {F_{D^j} \big( \gamma(t),
\dot{\gamma}(t) \big) dt} \leq d_{D_{\infty}} \left( ('0,-1), q \right) +
\epsilon
\]
which implies that
\[
\displaystyle\limsup_{j \rightarrow \infty} d_{D^j} \left( ('0,-1), \cdot \right)
\leq d_{D_{\infty}} \left( ('0,-1), \cdot \right).
\]

\noindent Note that $ B_{ D } ( p^j, R_j) \subset F^j(
\mathbb{B}^n)$. Since $ T^j \circ \phi $ are biholomorphisms and hence Kobayashi isometries, it
follows that
\begin{eqnarray}
B_{ D^j} \left( ('0,-1), R_j  \right) \subset T^j \circ \phi \circ F^j (\mathbb{B}^n). \label{4.0}
\end{eqnarray}
Since $ ( D_{\infty}, d_{D_{\infty}}) $ is complete, it is
possible to write
\begin{eqnarray}
D_{\infty} = \displaystyle \bigcup_{ \nu = 1} ^{\infty}
B_{D_{\infty}} \big( ('0,-1), \nu \big) \label{k10}
\end{eqnarray}
which is an exhaustion of $ D_{\infty}$ by an increasing union of
relatively compact domains. Now, consider
\[
\theta^j := \big( T^j \circ \phi \circ F^j \big)^{-1}: T^j \circ \phi \circ F^j (\mathbb{B}^n) \rightarrow
\mathbb{B}^n
\]
These mappings are evidently defined on an arbitrary compact
subset of $ D_{\infty} $ for large $j$ (cf. (\ref{k9}), (\ref{4.0}) and (\ref{k10})) and hence some subsequence
of $ \{ \theta^j \}$ converges to $ \theta: D_{\infty} \rightarrow
\overline{\mathbb{B}}^n$. Moreover, $ \theta \big( ('0,-1) \big) = ('0,0) $
together with the maximum principle shows that $ \theta :
D_{\infty} \rightarrow \mathbb{B}^n $. Finally observe that for
$w$ in a fixed compact set in $ D_{\infty} $,
\begin{eqnarray*}
| \psi \circ \theta (w) - w | & = & | \psi \circ \theta (w) -
\psi^j \circ \theta^j(w) | \\
& = & | \psi \circ \theta (w) - \psi \circ \theta^j (w) | + |
\psi \circ \theta^j(w) - \psi^j \circ \theta^j (w)| \\
& \rightarrow & 0 \ \mbox{as } j \rightarrow \infty
\end{eqnarray*}
This shows that $ \psi \circ \theta = id $. Similarly, it can be
proved that $ \theta \circ \psi = id$. This shows that $ 
D_{\infty} $ is biholomorphically equivalent to $ \mathbb{B}^n$.

\medskip

\noindent By composing with a suitable Cayley transform, if necessary, we may assume that there is a 
biholomorphism $ \tilde{\theta} $ from $ D_{\infty}$ onto the unbounded realization of the ball, namely to
\[
 \Sigma = \big\{ z \in \mathbb{C}^n : 2 \Re z_n + \abs{z_1}^2 + \abs{z_2}^2 + \ldots + 
\abs{z_{n-1}}^2 < 0 \big\}
\]
with the property that the cluster set of $ \tilde{\theta} $ at some point
$ ('0, \iota a) \in \partial D_{\infty} $ (for $a \in \mathbb{R} $) contains a  
point of $ \partial \Sigma $ different from the point at infinity on $ \partial \Sigma $. Then Theorem 2.1 of \cite{CP} ensures that $ \tilde{\theta} $ extends
holomorphically past the boundary of $ D_{\infty} $ to a neighbourhood  of $ ('0, \iota a) $. Furthermore, $ \tilde{\theta} $ 
extends biholomorphically across some point $ ('0, \iota a^0 ) \in \partial D_{ \infty} $.
To prove this claim, it suffices to show that the Jacobian of $ \tilde{\theta} $ does not vanish
identically on the complex plane 
\[
L = \big \{ ('0, \iota a) : a \in \mathbb{R} \big \} \subset \partial D_{\infty}.
\]
If the claim were false, then the Jacobian of $ \tilde{\theta} $ vanishes on the entire $ z_n $-axis, 
which intersects the domain $ D_{\infty} $. However, $ \tilde{\theta} $ is injective on $ D_{\infty} $,
and consequently, has nowhere vanishing Jacobian determinant on $ D_{\infty} $. This contradiction 
proves the claim.

\medskip

\noindent Next, note that the translations in the imaginary $ z_n $-direction 
leave $ D_{\infty} $ invariant. Therefore, we may assume that $ ('0, \iota a^0) $ is the origin and 
that $ \tilde{\theta} $ preserves the origin. Now recall that the Levi form is preserved under local biholomorphisms 
around a boundary point, thereby yielding the strong pseudoconvexity of $ \partial D_{\infty} $ near the origin.
Equivalently, the strong pseudoconvexity of $ p^0 \in \partial D $ follows. Hence the result.

\medskip

\noindent {\it Case (ii):} If $ p^0 $ were $ C^{\infty} $-convex of infinite type, by \cite{Z2}, there exists a sequence $ p^j $ in $ D $ 
converging to $ p^0 \in \partial D $ and affine isomorphisms $ A^j $ of 
$ \mathbb{C}^n $ so that the domains $ A^j(D) = D^j $ converge in the local Hausdorff topology to a convex domain $ D_{\infty} $ and $ A^j(p^j) \rightarrow 
0 \in D_{\infty}$. Moreover, it follows from Proposition 6.1 of \cite{Z1} that the limit domain $ D_{\infty} $ contains no complex affine lines and hence, it is Kobayashi complete.

\medskip

\noindent Since $ \lim_{j \rightarrow \infty} h_{D} (p^j) = 0 $ by assumption, there exists a sequence of positive real numbers $ R_j \rightarrow \infty $ and
a sequence of biholomorphic imbeddings $ F^j: \mathbb{B}^n \rightarrow D $ satisfying $ F^j (0)= p^j $ and $ B_D(p^j, R_j) \subset F^j( \mathbb{B}^n) $. Consider 
the maps
\[
 \psi^j := A^j \circ F^j : \mathbb{B}^n \rightarrow D^j.
\]
Note that $ A^j \circ F^j (0)  \rightarrow 0 \in D_{\infty} $. By Proposition 4.2 of \cite{Z1}, some subsequence of the sequence $ \{ A^j \circ F^j \} $ 
converges uniformly on compact sets of $ \mathbb{B}^n $ to a holomorphic mapping $ \psi : \mathbb{B}^n \rightarrow \overline{ D}_{\infty}  $. Since 
$ \psi(0)= 0 $, it follows that $ \psi ( \mathbb{B}^n ) \subset D_{\infty} $. 

\medskip

\noindent Then $ \psi $ is a biholomorphism from $ \mathbb{B}^n $ onto $ D_{\infty} $. This will be done in several steps. The first of these records the
stability of the infinitesimal Kobayashi metric, i.e., 
\begin{equation} \label{E1}
 F_{D^j} ( \cdot, \cdot) \rightarrow F_{D_{\infty}} (\cdot, \cdot)
\end{equation}
uniformly on compact sets of $ D_{\infty} \times \mathbb{C}^n $. The key step in proving the above assertion is to understand limits of 
holomorphic mappings $ f^j : \Delta \rightarrow D^j $ that almost realize $ F_{D^j} (\cdot, \cdot) $. Using Proposition 4.2 of \cite{Z1}, it is possible 
to pass to a subsequence of $ \{ f^j \} $ that converges to a holomorphic mapping $ f : \Delta \rightarrow D_{\infty} $ uniformly on compact sets of $ \Delta $. 
It follows that the limit map $ f $ provides a candidate in the definition of $ F_{D_{\infty}}( \cdot, \cdot) $. 

\medskip

\noindent The second step is to establish that
\begin{equation} \label{E21}
\displaystyle\limsup_{j \rightarrow \infty} d_{D^j} \left( A^j(p^j), \cdot \right)
\leq d_{D_{\infty}} \left( 0, \cdot \right), 
\end{equation}
which would imply that 
\begin{equation} \label{k11}
B_{D_{\infty}} \left( 0, R - \epsilon \right) \subset B_{D^j} \left( A^j(p^j), R \right), 
\end{equation}
for all $ R > 0 $ and all $j$ large and for each $ \epsilon > 0 $. To verify (\ref{E21}), fix $ q \in D_{\infty}$ as before and let $ \gamma : [0,1]
\rightarrow D_{\infty} $ be a piecewise $C^1$-smooth path in $
D_{\infty} $ such that $ \gamma(0) = 0, \gamma(1) = q$ and
\[
\int_0^1 { F_{D_{\infty}} \big( \gamma(t), \dot{\gamma}(t) \big)
dt} \leq d_{D_{\infty}} \left( 0, q \right) + \epsilon/2.
\]
Define $ \gamma^j: [0,1] \rightarrow \mathbb{C}^n $ by
\[
\gamma^j(t) = \gamma(t) + A^j(p^j)(1-t).
\]
Since the trace of $ \gamma $ is relatively compact in $
D_{\infty}$ and $ A^j (p^j) \rightarrow 0 $, it follows that the trace
of $\gamma $ is contained uniformly relatively compactly in $ D^j
$ for all large $j$. Note that $ \gamma^j(0) = A^j(p^j) $ and $ \gamma^j(1) = q $. In addition, $ \gamma^j
\rightarrow \gamma $ and $ \dot{\gamma}^j \rightarrow \dot{\gamma}
$ uniformly on $ [0,1]$. Appealing to (\ref{E1}) yields
\[
\int_0^1 {F_{D^j} \big( \gamma^j(t), \dot{\gamma}^j(t) \big) dt}
\leq \int_0^1 { F_{D_{\infty}} \big( \gamma(t), \dot{\gamma}(t)
\big) dt} + \epsilon/2 \leq d_{D_{\infty}} ( 0, q) + \epsilon.
\]
Therefore,
\[
d_{D^j} \left( A^j(p^j), q \right) \leq \int_0^1 {F_{D^j} \big( \gamma(t),
\dot{\gamma}(t) \big) dt} \leq d_{D_{\infty}} \left( 0, q \right) +
\epsilon
\]
which implies that
\[
\displaystyle\limsup_{j \rightarrow \infty} d_{D^j} \left( A^j(p^j), \cdot \right)
\leq d_{D_{\infty}} \left( 0, \cdot \right),
\]
thereby, establishing (\ref{E21}). 

\medskip

\noindent Next, recall that $ B_{ D } ( p^j, R_j) \subset F^j(\mathbb{B}^n) $. Since $ A^j $ are biholomorphisms and hence Kobayashi isometries, it
follows that
\begin{eqnarray}
B_{ D^j} \left( A^j(p^j) , R_j  \right) \subset A^j \circ F^j (\mathbb{B}^n). \label{5.0}
\end{eqnarray}
Since $ ( D_{\infty}, d_{D_{\infty}}) $ is complete, it is
possible to write
\begin{eqnarray}
D_{\infty} = \displaystyle \bigcup_{ \nu = 1} ^{\infty}
B_{D_{\infty}} \big( 0, \nu \big). \label{k13}
\end{eqnarray}
Now, consider the mappings
\[
\theta^j := \big( A^j \circ F^j \big)^{-1}: A^j \circ F^j (\mathbb{B}^n) \rightarrow
\mathbb{B}^n.
\]
It follows from  (\ref{k13}), (\ref{k11}) and (\ref{5.0}) that the mappings $ \theta^j$ are defined on any arbitrary compact
subset of $ D_{\infty} $ for large $j$. In particular, $ \{\theta^j \} $ is normal. Let $ \theta: D_{\infty} \rightarrow
\overline{\mathbb{B}}^n$ be a holomorphic limit of some subsequence
of $ \{ \theta^j \}$. Since $ \theta \big( 0 \big) = 0 $, it follows that $ \theta :
D_{\infty} \rightarrow \mathbb{B}^n $. The final step is to note that $ \psi \circ \theta = id $ and $ \theta \circ \psi = id$ as before. A consequence 
of all of this is that $ 
D_{\infty} $ is biholomorphically equivalent to $ \mathbb{B}^n$.

\medskip

\noindent Now we seek an contradiction. If $ p^0 \in \partial D $ were $ C^{\infty} $-convex of infinite type, then the limit domain $ \partial D_{\infty} $ 
contains a non-trivial complex affine disc (cf. Proposition 6.1, \cite{Z1}) and hence, it follows from Theorem 3.1 of \cite{Z1} that
$ \left(D_{\infty}, d_{D_{\infty}} \right) $ is not Gromov hyperbolic. Since $ \left( \mathbb{B}^n, d_{\mathbb{B}^n} \right) $ is Gromov hyperbolic
and Gromov hyperbolicity is an isometric invariant, it follows that $ \left(D_{\infty}, d_{D_{\infty}} \right) $ and  $ \left( \mathbb{B}^n, d_{\mathbb{B}^n} \right) $
cannot be isometric. This is a contradiction since $ D_{\infty} $ is biholomorphic to $ \mathbb{B}^n $. Hence, the boundary point $ p^0 $ has 
to be of finite type and we are in Case (i). \qed

\section{Proof of Theorem 1.2}

\noindent Proving Theorem 1.2 involves verifying the following two results:

\begin{prop} \label{B1}
Let $D \subset \mbb C^n$ be a bounded domain with $p^0 \in \partial D$. Assume that $\pa D$ is
$C^{\infty}$-smooth and $h$-extendible near $p^0$. Then $\partial D$ is strongly pseudoconvex near $p^0$ if 
$h_D(z) \ra 0$ as $z \ra p^0$. 
\end{prop}

\begin{prop} \label{B2}
Let $D \subset \mbb C^n$ be a bounded domain with $p^0 \in \partial D$. Assume that $\pa D$ is
$C^{\infty}$-smooth and $h$-extendible near $p^0$. Then $\partial D$ is strongly pseudoconvex near $p^0$ if $s_D(z) \ra 1$ as $z \ra p^0$. 
\end{prop}

\noindent \textit{Proof of Proposition \ref{B1}.} Let $ p^j $ be a sequence of points in $ D $ converging to $ p^0 $
along the inner normal to $ \partial D $ at $ p^0 $. The boundary point $ p^0 $ is a local peak point by \cite{Yu-1994} and 
hence the Fridman's invariant function can be localized near $ p^0 $ (refer Proposition 3.4 of \cite{MV}). It follows that $ \lim_{j \rightarrow \infty} h_{U \cap D} (p^j, \mathbb{B}^n ) = 0 $ for any neighbourhood $ U $ of $ z^0 $. As a consequence, there exists a sequence of positive real numbers $ R_j \rightarrow \infty $ and
a sequence of biholomorphic imbeddings $ F^j: \mathbb{B}^n \rightarrow U \cap D $ satisfying $ F^j (0)= p^j $ and $ B_{U \cap D}(p^j, R_j) \subset F^j( \mathbb{B}^n) $.

\medskip

\noindent Let us quickly recall the local geometry of $h$-extendible domains. Firstly, if $ \mathcal{M}(\partial D, p^0) = (1, m_2, \ldots, m_n) $ denotes the Catlin's
multitype of $ \partial D $ near $ p^0 $, then there is an automorphism $ \Phi $ of $ \mathbb{C}^n $ such that 
such that $ \Phi(p^0) = ('0, 0) $, $ \Phi (p^j) = ('0, - \| p^0 - p^j \|) $ for each $ j$ and the defining function for the domain $ \Phi(D) $ can be
expanded near the origin as 
\begin{alignat*}{3}
2 \Re z_n + P('z, ' \overline{z}) + R(z),                                        
\end{alignat*}
where $ P('z, ' \overline{z}) $ is a $ (1/m_n , 1/m_{n-1} , \ldots , 1/m_2 ) $ homogeneous polynomial of weight one
which is plurisubharmonic and does not contain pluriharmonic terms and $ R $ satisfies
\[
|R(z)| \lesssim \left( |z_1 |^{m_n} + |z_2 |^{m_{n-1}} + \ldots + |z_n | \right)^{\gamma}
\]
for some $ \gamma > 1 $. Let $ T^j $ be the dilation defined by
\[
T^j( z_1, z_2, \ldots, z_{n-1}, z_n ) = \left( \delta_j^{-1/{m_n}} z_1, \delta_j^{-1/{m_{n-1}}} z_2, \ldots, \delta_j^{-1/m_2} z_{n-1}, \delta_j^{-1} z_n \right),
\] 
where $ \delta_j = \|p^0 - p^j \| $ for each $j$. Note that $ T^j \big( ('0, - \delta_j) \big) = ('0, -1) $ for all $j$ while the domains 
$ D^j := T^j \circ \Phi(U \cap D) $ 
converge in the Hausdorff sense to
\[
D_{ \infty} = \big\{ z \in \mathbb C^n : 2 \Re z_n  + P('z, ' \overline{z}) < 0 \big\}.
\]
On the other hand, 
by Theorem 4.7 of \cite{Yu-1995}, there is an $h$-extendible model 
\[
 \Omega_0 = \big\{ z \in \mathbb{C}^n : 2 \Re z_n + Q('z, ' \overline{z}) < 0 \big\},
\]
where $ Q('z, ' \overline{z}) $ is a $ (1/m_n , 1/m_{n-1} , \ldots , 1/m_2 ) $ homogeneous polynomial (of weight one
which is plurisubharmonic but not pluriharmonic), such that if $ U $ is a small neighbourhood of $ p^0 $, then
\[
 \Phi(U \cap D) \subset \Omega_0. 
\]
Consequently, the scaled domains $ D^j $ satisfy
\begin{eqnarray*}
 D^j = T^j \circ \Phi(U \cap D) \subset T^j(\Omega_0)
\end{eqnarray*}
for all $j $ large. The homogenity of $ Q('z, ' \overline{z}) $ with weight one as described above implies that 
\begin{equation*}
 Q \left(  \delta_j^{1/{m_n}} z_1, \delta_j^{1/{m_{n-1}}} z_2, 
\ldots, \delta_j^{1/m_2} z_{n-1} \right) = \delta_j Q ('z, ' \overline{z} )
\end{equation*}
for each $j $. In other words, the mappings $ T^j $ leave $ \Omega_0 $ invariant and hence
\begin{eqnarray} \label{E15}
 D^j = T^j \circ \Phi(U \cap D) \subset \Omega_0.
\end{eqnarray}
Furthermore, observe that $ D_{\infty} $ and $ \Omega_0 $ are complete hyperbolic (cf. \cite{Yu-1994}) and hence both $ D_{\infty} $ and $ \Omega_0 $
are taut.

\medskip 
\noindent Let us consider the holomorphic mappings
\[
\psi^j := T^j \circ \phi \circ F^j : \mathbb{B}^n \rightarrow D^j.
\]
Note that $ T^j \circ \phi \circ F^j \big( ('0,0) \big) = ('0,-1)$ for each $ j $. As the scaled domains $ D^j $ are contained in a taut domain $ \Omega_0 $, 
the sequence $ \{ T^j \circ \phi \circ F^j  \} $ admits
a subsequence, that will still be denoted by the same indices, that
converges uniformly on compact sets of $ \mathbb{B}^n $ to a
holomorphic mapping $ \psi : \mathbb{B}^n \rightarrow D_{\infty} $.

\medskip

\noindent Here, too, it turns out that $ \psi $ is a biholomorphism from $ \mathbb{B}^n $ onto $ D_{\infty} $. To establish this, first note that 
\begin{eqnarray} \label{E4}
 F_{D^j} ( \cdot, \cdot) \rightarrow F_{D_{\infty}} ( \cdot, \cdot) 
\end{eqnarray}
uniformly on compact sets of $ D_{\infty} \times \mathbb{C}^n $. The above assertion 
can be proved using the fact that the scaled domains $ D^j $ are all contained in $ \Omega_0 $ for large $ j $ (cf. (\ref{E15})).
Once the stability of the infinitesimal metric is understood (i.e., (\ref{E4})), by similar arguments as in the proof of Theorem 1.1(i), it follows that
for each $ \epsilon > 0 $,
\begin{equation}
B_{D_{\infty}} \left( ('0,-1), R - \epsilon \right) \subset B_{D^j} \left( ('0,-1), R \right)
\label{E2}
\end{equation}
for all $ R > 0 $ and all $j$ large. 
\medskip

\noindent Recall that $ B_{U \cap D } ( p^j, R_j) \subset F^j(
\mathbb{B}^n)$. Since $ T^j \circ \phi $ are biholomorphisms and hence Kobayashi isometries, it
follows that
\begin{eqnarray}
B_{ D^j} \left( ('0,-1), R_j  \right) \subset T^j \circ \phi \circ F^j (\mathbb{B}^n). \label{E5}
\end{eqnarray}
Exploiting the Kobayashi completeness of the limit domain $ D_{\infty}$, it is
possible to write
\begin{eqnarray}
D_{\infty} = \displaystyle \bigcup_{ \nu = 1} ^{\infty}
B_{D_{\infty}} \big( ('0,-1), \nu \big). \label{E6}
\end{eqnarray}
Now, consider the mappings
\[
\theta^j := \big( T^j \circ \phi \circ F^j \big)^{-1}: T^j \circ \phi \circ F^j (\mathbb{B}^n) \rightarrow
\mathbb{B}^n
\]
defined on an arbitrary compact
subset of $ D_{\infty} $ for large $j$ (cf. (\ref{E2}), (\ref{E5}) and (\ref{E6})) and hence some subsequence
of $ \{ \theta^j \}$ converges to $ \theta: D_{\infty} \rightarrow
\overline{\mathbb{B}}^n$. Moreover, $ \theta \big( ('0,-1) \big) = ('0,0) $
together with the maximum principle shows that $ \theta :
D_{\infty} \rightarrow \mathbb{B}^n $. Then it can be checked that $ \psi \circ \theta = id $ and $ \theta \circ \psi = id$ and hence $ 
D_{\infty} $ is biholomorphically equivalent to $ \mathbb{B}^n$.

\medskip
\noindent Finally, we are ready to prove the theorem. Using similar arguments as in the proof of Theorem 1.1(i), it is 
possible to show that there is a biholomorphism $ \tilde{\theta} $ from $ D_{\infty}$ onto the unbounded realization of the ball, namely to
\[
 \Sigma = \big\{ z \in \mathbb{C}^n : 2 \Re z_n + \abs{z_1}^2 + \abs{z_2}^2 + \ldots + 
\abs{z_{n-1}}^2 < 0 \big\},
\]
with the property that $ \tilde{\theta} $ extends
biholomorphically past the boundary of $ D_{\infty} $ to a neighbourhood  of the origin. Also, $ \tilde{\theta} \big( ('0,0) \big) = 
\big( ('0,0) \big) $. Since the Levi form is preserved under local biholomorphisms 
around a boundary point, it follows that $ \partial D_{\infty} $ must be strongly pseudoconvex near the origin.
In particular, $ m_2= \ldots = m_n =2 $ and $ P('z, ' \overline{z}) = |z_1|^2 + \ldots + |z_n|^2  $ which gives the strong pseudoconvexity 
of $ p^0 \in \partial D $. Hence the result. 
\qed

\medskip

\noindent \textit{Proof of Proposition \ref{B2}.} Let $ p^j $ be a sequence of points in $ D $ converging to $ p^0 $ along 
the inner normal to $ \partial D $ at $ p^0 $. Since $ \lim_{j \rightarrow \infty} s_{D} (p^j) = 1 $ by assumption, there 
exists a sequence of positive real numbers $ R_j \rightarrow 1 $ and
a sequence of biholomorphic imbeddings $ F^j: D \rightarrow \mathbb{B}^n $ satisfying $ F^j(p^j)=0 $ and $ {B}^n(0, R_j) \subset F^j( D) $.

\medskip

\noindent Let us adapt here the method of uniform scaling. Let $ U $ be a neighbourhood of $ p^0 \in \partial D $ and the mappings $ T^j \circ \Phi $ 
be as described in the the proof of Proposition \ref{B1}. So that the domains $ D^j = T^j \circ \Phi (U \cap D) $ converge in the Hausdorff sense to
\[
D_{ \infty} = \big\{ z \in \mathbb C^n : 2 \Re z_n + P('z, ' \overline{z}) < 0 \big\}.
\]
Here, $ P('z, ' \overline{z}) $ is a $ (1/m_n , 1/m_{n-1} , \ldots , 1/m_2 ) $ homogeneous polynomial of weight one that coincides with the polynomial
of same degree in the homogeneous Taylor expansion of the defining function for $ \Phi(D) $ 
near the origin and $ (1, m_2, \ldots, m_n) $ is the Catlin's multitype of $ \partial D $ near $ p^0 $.

\medskip 
\noindent Consider the maps
\[
\psi^j := F^j  \circ \left( T^j \circ \phi \right)^{-1} : D^j \rightarrow F^j(D) \subset \mathbb{B}^n.
\]
Note that $ F^j  \circ \left( T^j \circ \phi \right)^{-1} \big( ('0,-1) \big) = ('0,0)$ for each $ j $. These mappings are defined on an arbitrary compact
subset of $ D_{\infty} $ for large $j$ and hence some subsequence
of $ \{ \psi^j \}$ converges to $ \psi: D_{\infty} \rightarrow
\overline{\mathbb{B}}^n$. Since $ \psi \big( ('0,-1) \big) = ('0,0) $, it follows that $ \psi (D_{\infty}) \subset \mathbb{B}^n $.

\medskip
\noindent The claim is that $ \psi $ is a biholomorphism from $ D_{\infty} $ onto $ \mathbb{B}^n $. To begin with, let $ K $ be 
an arbitrary compact set of $ \mathbb{B}^n $ containing the origin. Since $ {B}^n (0, R_j) \subset F^j(D) $ and 
$ R_j \rightarrow 1 $, the mappings $ (F^j)^{-1} $ are defined on $ K $ for $ j $ large. Since $ (F^j)^{-1} \big(('0,0) \big) = p^j \rightarrow p^0 \in \partial D $ and 
the boundary point $ p^0 $ is a local peak point, it follows from the attraction property (see for instance Lemma 2.1.1 of \cite{Gaussier-1999}) that
\[
 (F^j)^{-1} (K) \subset U \cap D 
\]
for all $ j $ large. Therefore
\[
\left( T^j \circ \phi \right) \circ \left(F^j \right)^{-1} (K) \subset \left( T^j \circ \phi \right) (U \cap D) =  D^j
\]
and hence the scaling sequence
\[
 \theta^j := \left( T^j \circ \phi \right) \circ \left(F^j \right)^{-1}
\]
maps $ K $ injectively into $ D^j $ for all $ j $ large. On the other hand, 
by (\ref{E15}), there is a taut domain $ \Omega_0 $ such that $ D^j \subset \Omega_0 $ 
for all $j $ large. It follows that  $ \{ T^j \circ \phi \circ \left(F^j \right)^{-1}|_K  \} $ forms a normal family. Let $ \theta : K \rightarrow 
\overline{D}_\infty $ be a holomorphic limit of some subsequence of $ \{ T^j \circ \phi \circ \left(F^j \right)^{-1} \} $. Since $ K $ is an arbitrary
compact subset of $ \mathbb{B}^n $, $ \theta $ is defined on whole of $ \mathbb{B}^n $. As noted earlier, $ T^j \circ \phi \circ \left(F^j \right)^{-1} 
\big( ('0,0) \big) = ('0,-1)$ for each $ j $ and so $ \theta \big( ('0,0) \big) = ('0, -1) $. But $ ('0, -1) \in D_{\infty} $ and $ D_{\infty} $ is open, so 
$ \theta (\mathbb{B}^n ) \subset D_{\infty} $. 

\medskip
\noindent Now an argument similar to the one employed in Theorem 1.1 shows that  $ \psi \circ \theta = id $ and $ \theta \circ \psi = id$, which in
turn implies that $ D_{\infty} $ is biholomorphically equivalent to $ \mathbb{B}^n$. As in the proof of Proposition \ref{B1}, the strong pseudoconvexity 
of $ p^0 \in \partial D $ follows.
\qed

\medskip

\noindent {\it Concluding Remarks:} While the construction of $h_D$ and $s_D$ are completely dual to each other, the exact relation
between them is unclear from their definitions. It would be interesting and useful to clarify this. 

\medskip

The ratio of the Carath\'{e}odory--Eisenmann and the Kobayashi volume forms is another biholomorphic invariant
that has been studied in the past. It is known that it is at most $1$ everywhere and that if it equals $1$ at an
interior point, then the domain is biholomorphic to $\mbb B^n$. Is there an analog of this result for the ratio $s_D/h_D$?
 
\medskip

\noindent {\it Acknowledgements:} The authors would like to thank A. Zimmer for pointing out a gap in an earlier version of this article.


\end{document}